\documentclass{amsart}
\usepackage[utf8x]{inputenc}
\usepackage{bbm}
\usepackage{amsmath,amsthm,amsfonts,amssymb,amscd,mathrsfs}
\usepackage{psfrag,graphicx}

\usepackage{epic,eepic}
\usepackage{color}
\usepackage{ebezier}
\usepackage{enumerate}

\newtheorem{Thm}{Theorem}[section]
\newtheorem{Lem}[Thm]{Lemma}
\newtheorem{Prop}[Thm]{Proposition}

\newtheorem{Def}[Thm] {Definition}

\title[Liv\v{s}ic measurable rigidity]{Liv\v{s}ic measurable rigidity theorem for $\mathcal{C}^1$ generic  volume-preserving Anosov  systems}
\date{September. 2014}
\author{Yun Yang$^{*}$}
 \address{ $^{*}$School of Mathematical Sciences,
Peking University, Beijing 100871, China; 
Email: summeryangyun@gmail.com}
\begin{document}
\maketitle
\begin{abstract}
In this paper, we prove that for $\mathcal{C}^1$ generic volume-preserving  Anosov diffeomorphisms  of  a  compact   Riemannian manifold,  Liv\v{s}ic measurable rigidity theorem holds. We also prove that for $\mathcal{C}^1$ generic volume-preserving Anosov flows  of a  compact   Riemannian manifold,  Liv\v{s}ic measurable rigidity theorem holds.
\end{abstract}

\section{Introduction}
  
Let $T: M\rightarrow M$  be a diffeomorphism on a compact
 Riemannian manifold $M$.  We consider a
 {\bf cocycle} $\mathcal{A}: \mathbb{Z}\times M\rightarrow \mathbb{R}$; that is,  a map
satisfying the cocycle relation
$$\mathcal{A}(n_1+n_2,x)=\mathcal{A}(n_1, T^{n_2}(x))+\mathcal{A}(n_2,x),$$
for every $n_1,n_2\in \mathbb{Z}$ and every $x\in M$. 
Following the definition in cohomological algebra, we call a cocycle $\mathcal{A}$  {\bf  a coboundary} if it satisfies the
cohomological equation:
\begin{eqnarray}\label{coboundary}\mathcal{A}(n,x)=\Phi(T^{n}(x))-\Phi(x),
\end{eqnarray}
where $\Phi:M\rightarrow \mathbb{R}$ is a function.  Furthermore, two cocycles are
called  cohomologous if their difference is a coboundary.

It is easy to see that  coboundary $\mathcal{A}$ must have
{\bf trivial periodic data}, i.e.
\begin{eqnarray}\label{POO}\mathcal{A}(n,x)=0, \,\,\,\,\forall x\in M,  \,\,\,T^{n}(x)=x.
\end{eqnarray}
One has three  natural questions to propose.

1. Is the necessary condition, trivial periodic data,   also a
sufficient condition? 

2. {\bf Measurable rigidity:} If the cocycle $\mathcal{A}: \mathbb{Z}\times M\rightarrow \mathbb{R}$ is H\"{o}lder continuous, can we get a H\"{o}lder continuous solution  $\Phi$  to equation (\ref{coboundary})  from a  measurable solution? 

3. {\bf Higher regularity:} If the cocycle $\mathcal{A}: \mathbb{Z}\times M\rightarrow \mathbb{R}$ is $\mathcal{C}^r$ for some $1\leq r\leq \infty$ or $r=\omega$, is  a continuous solution to equation (\ref{coboundary}) also $\mathcal{C}^r$?

Liv\v{s}ic took the lead in considering  these three questions  for the case when $f$ is a transitive Anosov diffeomorphism on a compact     Riemannian manifold $M$            
   \cite{livsic,livsic 1}.  Thus, we call results answering the above questions Liv\v{s}ic theorems.  Current research is usually concerned with two variations on this subject, namely altering the base system $T$ and altering the group $\mathbb{R}$.  Some of the highlights are \cite{R.D.L,R.D.L2,R.D.L3,NP,NT,NT2,P,PP,PW, K}. 
The following theorems are some classical results. 
\begin{Thm}\cite{livsic, livsic 1,C.W, HK}\label{livsic for Anosov} Let $T:M\rightarrow M$ be a $\mathcal{C}^1$ transitive Anosov  diffeomorphism on a compact    Riemannian manifold $M$ and let $\phi:M\rightarrow \mathbb{R}$ be a H\"{o}lder continuous function.
\begin{enumerate}
 \item[$(1)$] {\bf Existence of solutions.}  $\phi=\Phi(T)-\Phi$
has a H\"{o}lder continuous solution  $\Phi$ if and only if $\sum_{x\in\mathcal{O}}\phi(x)=0$, for every $T$-periodic orbit $\mathcal{O}$.
\item[$(2)$] {\bf Measurable rigidity.}  For any Gibbs measure $\mu$ with H\"{o}lder continuous potential,  if there exists a $\mu$-measurable solution $\Phi$ to $\phi=\Phi(T)-\Phi$, then there is a continuous
solution $\Psi$, with $\Phi=\Psi, a.e. \mu.$ 
 \end{enumerate}
\end{Thm}
 \begin{Thm}\cite{livsic, livsic 1,C.W, HK}\label{livsic for flow}  Let $\{T^t\}$ be a $\mathcal{C}^1$  transitive  Anosov flow  on a compact     Riemannian manifold $M$ generated by the vector field $\xi$ and let $\phi:M\rightarrow \mathbb{R}$ be a H\"{o}lder continuous function. 
\begin{enumerate}
 \item[$(1)$]{\bf Existence of solutions.} 
$\phi=\Phi'_{\xi}$ for a H\"{o}lder continuous function $\Phi$ differentiable along the flow if and only if $\int^{t_p}_0\phi(T^s(p))ds=0$ for every periodic point $p$ with period $t_p$. 
\item[$(2)$] {\bf Measurable rigidity.}  For any Gibbs measure $\mu$ with H\"{o}lder continuous potential,  if there exists a $\mu$-measurable solution $\Phi$ differentiable along the flow such that  $\phi=\Phi'_{\xi}$ $\mu$-almost everywhere, then there is a continuous
solution $\Psi$, with $\Phi=\Psi, a.e. \mu.$ 
\end{enumerate}
\end{Thm}

In this paper, we are only concerned with measurable rigidity. The proof of measurable rigidity in \cite{C.W} is 
based on the Markov partitions and Liv\v{s}ic type theorems for cocycles over shifts of finite type \cite{PP}, which   depend heavily on the equipment of Gibbs measures.  For the definition of Gibbs measures, we refer the reader to  a classical and short book  \cite{B} by Bowen.
For other measures, measurable rigidity may not hold. 

In this paper, we consider the measurable rigidity for  the special measure, volume measure $m$.  It is known that for $\mathcal{C}^2$ volume-preserving Anosov diffeomorphisms, the volume measure is a Gibbs measure with the H\"{o}lder continuous potential $$\varphi=-\log \det(DT|E^u).$$
However,  the volume measure  for $\mathcal{C}^1$ volume-preserving Anosov diffeomorphism may not be a Gibbs measure  with  H\"{o}lder continuous potential.   

Under a $\mathcal{C}^1$ generic hypothesis, we have the following result.
\begin{Thm}\label{Glivsic}There exists a residual subset $\mathcal{G}$ of $\mathcal{C}^1$ Anosov volume-preserving diffeomorphisms  on a compact   Riemannian   manifold $M$ such that for any $T\in\mathcal{G}$ and any 
H\"{o}lder 
continuous function $\phi:M\rightarrow \mathbb{R}$ ,
the following three conditions are equivalent:
\begin{enumerate}\item[$(1)$]$\sum_{x\in \mathcal{O}}\phi(x)=0$, for every $T$-periodic orbit $\mathcal{O}$,
 \item[$(2)$] $\phi(x)=\Phi(T(x))-\Phi(x)$ has a continuous solution,
\item[$(3)$] $\phi(x)=\Phi(T(x))-\Phi(x),a.e.$ for some measurable function $\Phi$.
\end{enumerate}
\end{Thm}

 We also get a parallel result for Anosov flows. 

 \begin{Thm}\label{Flow}There exists a residual subset $\mathcal{G}$ of   $\mathcal{C}^1$  Anosov volume-preserving flows on a compact    Riemannian manifold $M$ such that for any  flow $\{T^t\}\in \mathcal{G}$  and  any H\"{o}lder 
continuous function $\phi:M\rightarrow \mathbb{R}$,  
the following three conditions are equivalent:
\begin{enumerate}\item[$(1)$] $\int^{t_p}_0\phi(T^s(p))\,ds=0$ for every periodic point $p$ with period $t_p$, 
\item[$(2)$] $\phi=\Phi'_{\xi}$ for a H\"{o}lder continuous function $\Phi$ differentiable along the flow,
\item[$(3)$] $\phi=\Phi'_{\xi}$ almost everywhere for a measurable function $\Phi$ differentiable along the flow.
\end{enumerate}
\end{Thm}

\section*{Acknowledgements}
The author would like to thank Artur Avila and Amie Wilkinson for their useful suggestions. The author also would like to thank Jinxin Xue and Rachel Vishnepolsky for their careful reading. The author also would like to thank  China Scholarship Council for its financial support. 
\section{Preliminaries}
\subsection{Anosov diffeomorphisms}
Assume $M$ to be a compact   Riemannian manifold.  Recall that a diffeomorphism $T:M\rightarrow M$ is called  Anosov if there  is a $T$-invariant splitting $$TM=E^s\oplus E^u$$ and constants $C,\rho<1$, such that 
$$\forall v\in E^s, \Vert  DT^nv\Vert \leq C\rho^n\Vert v\Vert , $$
$$\forall v\in E^u, \Vert  DT^{-n}v\Vert \leq C\rho^n\Vert v\Vert . $$ 
Now we formulate the Central Limit Theorem for $\mathcal{C}^2$ volume-preserving  Anosov diffeomorphisms. Its proof involves the construction of  Markov partition of Anosov diffeomorphisms and the corresponding statistical property of  subshifts of finite type.  
\begin{Thm}[Central Limit Theorem]\cite{B}\label{Central Limit Theorem}Let $T$ be a  $\mathcal{C}^2$  Anosov volume-preserving diffeomorphism on 
compact    Riemannian manifold
 $M$. Let $m$ be the volume measure on $M$. Let $\phi$ be a  H\"{o}lder continuous function on $M$ with no measurable solution $\Phi$ to the equation: $$\phi(x)-\int \phi(x)\,dx=\Phi(T(x))-\Phi(x).$$
  Then $\phi$ satisfies the Central Limit Theorem with respect to $T$, i.e.
there exists a constant $\sigma>0$ such that for any $-\infty<\alpha<+\infty$,
$$\lim_{n\rightarrow+\infty}m\left\{x\in M: \frac{\sum_{i=0}^{n-1}\phi(T^i(x))-n\overline{\phi}}{\sigma \sqrt{n}}<\alpha\right\}=
\frac{1}{\sqrt{2\pi}}\int^{\alpha}_{-\infty} e^{-\frac{1}{2}u^2}\,du,$$
where $\overline{\phi}=\int_M\, \phi(x) dm.$
What's more, $\sigma^2=\lim_{n\rightarrow+\infty}\frac{\int(\phi_n)^2\,dx}{n}$,
 where $\phi_n=\sum_{i=0}^{n-1}\\(\phi(T^i(x))-\overline{\phi})$ and
$\sigma$ is called the variance with respect to $\phi.$
\end{Thm}
\subsection{Anosov flows}
Let us recall the definition of Anosov flow first. Let $M$ be a compact   Riemannian manifold. Let $T:\mathbb{R}\times M\rightarrow M$ be a $\mathcal{C}^1$ flow on $M$ generated by the vector field $\xi=\frac{d}{dt}(T^t)|_{t=0}$ where $T^t(\cdot)$ denotes $T(t,\cdot)$. Flow $T:\mathbb{R}\times M\rightarrow M$ is called Anosov if there is a continuous splitting $TM=E^{+}\oplus E^0\oplus E^{-}$ with $E^0$ spanned by $\xi$ and there are positive constants $c_1,c_2$ and $\gamma$ such that
$$\Vert  D(T^t)(\eta)\Vert \geq c_1\cdot e^{t\gamma}\cdot \Vert \eta\Vert , \,\forall\,  \eta\in E^{+} \text{ and } t\geq 0,$$
$$\Vert  D(T^t)(\eta)\Vert \leq c_2\cdot e^{-t\gamma}\cdot \Vert \eta\Vert ,\, \forall\,  \eta\in E^{-} \text{ and } t\geq 0.$$
We use the notation $\{T^t\}$ as the flow $T(t,\cdot)$ in this paper. We give the $\mathcal{C}^r$ topology of flows in the following definition.  
\begin{Def}[ $\mathcal{C}^r$ Topology for Flows] Let $\mathcal{F}^r(M)$ be the space of $\mathcal{C}^r$-flows on $M$. Every flow $\{T^t\}\in \mathcal{F}^r(M)$ restricts to a $\mathcal{C}^r$ map $T^{[t_0]}:[0,t_0]\times M\rightarrow M.$ Since $[0,t_0]\times M$ is compact, we may take the usual $\mathcal{C}^r$ topology on $\mathcal{C}^r$ maps $[0,t_0]\times M\mapsto M$, and thereby define a $\mathcal{C}^r$ topology on $\mathcal{F}^r(M)$. Using the one parameter group property of flows, it is easy to see that the $\mathcal{C}^r$ topology we have defined on $\mathcal{F}^r(M)$ is independent of $t_0>0$. 
\end{Def}
In the interest of proving Theorem \ref{Flow}, we will also use the Central Limit Theorem for $\mathcal{C}^2$ volume-preserving Ansov flows, which is proved by taking advantage of the Markov partition for Anosov flows.  
\begin{Thm}[Central Limit Theorem for Anosov flows]\cite{R}  \label{C2}Let $\{T^t\}$ be a   $\mathcal{C}^2$ Anosov volume-preserving flow on a compact    Riemannian manifold $M$ generated by vector field $\xi$ and let $\phi:M\rightarrow \mathbb{R}$ be a H\"{o}lder continuous function on $M$. Let $m$ be the volume measure on $M$. If there is no measurable function $\Phi:M\rightarrow \mathbb{R}$ differentiable along the flow $\{T^t\}$ such that 
$$\phi=\Phi'_{\xi},\,\,\, a. e. $$
then $\phi$ satisfies the Central Limit Theorem relative to $\{T^t\}$, i.e. 
  there exists a constant $\sigma>0$ such that for any $-\infty<\alpha<+\infty$ 
$$\lim_{t\rightarrow +\infty}m\left\{x\in M: \frac{\int^t_0\,(\phi(T^s(x))-\overline{\phi})\,ds}{\sigma \sqrt{t}}<\alpha\right\}=\frac{1}{\sqrt{2\pi}}\int^{\alpha}_{-\infty} e^{-\frac{1}{2}u^2}du$$
where $\overline{\phi}=\int_{M}\phi(x)\,dm.$ Moreover,
$$\sigma^2=\lim_{t\rightarrow+\infty}\frac{\int_{M}(\phi_t(x))^2 \,dm}{t}$$
where $\phi_t$ denotes $\int_{0}^{t}(\phi(T^s(x)-\overline{\phi})\,ds$ and $\sigma$ is called the variance to function $\phi$.  
\end{Thm}
We formulate the definition of cocycles for Anosov flows, introducing some useful terminology along the way.   In particular, let $T:\mathbb{R}\times M\rightarrow M$ be a $C^1$ Anosov flow,  a continuous map $\mathcal{A}:{\mathbb{R}}\times M \rightarrow \mathbb{R}$ is {\bf a  cocycle} over $T$ if 
$$\mathcal{A}(t_1+t_2,x)=\mathcal{A}(t_1, T^{t_2}(x))+\mathcal{A}(t_2,x)$$
for every $t_1,t_2\in \mathbb{R}$ and every $x\in M$.  Cocycle  $\mathcal{A}:\mathbb{R}\times M\rightarrow \mathbb{R}$ is  {\bf a coboundary} if there exists a map $\Phi:M\rightarrow \mathbb{R}$ such that 
\begin{eqnarray}\label{coc}\mathcal{A}(t,x)=\Phi(T^t(x))-\Phi(x). 
\end{eqnarray}

From now on, we only consider cocycles which are  differentiable along the flow. Namely,  $\mathcal{A}(t,p)$ is a $\mathcal{C}^1$ function of $t$ for all $p\in M$. A cocycle $\mathcal{A}:\mathbb{R}\times M\rightarrow \mathbb{R}$ is called  H\"{o}lder continuous of exponent $\alpha\in(0,1)$ if the map 
$$x\mapsto \lim_{t\rightarrow 0}\frac{1}{t}\mathcal{A}(t,x)$$
is {\bf H\"{o}lder continuous} of exponent $\alpha$. Thus  not only the cocycle  $\mathcal{A}$ is differentiable along the flow, but also the derivative of $\mathcal{A}$ along the flow are   H\"{o}lder continuous functions on the manifold $M$. 

There is a natural bijection between cocycles and functions on $M$.    A cocycle $\mathcal{A}(t,x)$ is said to be based on a function $\phi:M\rightarrow \mathbb{R}$ if 
$$\mathcal{A}(t,x)=\int_{0}^t\phi(T^s(x))\,ds. $$
The function $\phi:M\rightarrow \mathbb{R}$ is {\bf the infinitesimal generator} $\xi(\mathcal{A})$ of cocycle $\mathcal{A}$.  The existence of this generator is due to the differentiability of the cocycle $\mathcal{A}$ along the flow.  The cocycle $\mathcal{A}$ is   H\"{o}lder continuous with exponent $\alpha$ if and only if the function $\phi:M\rightarrow \mathbb{R}$ is a H\"{o}lder continuous function with exponent $\alpha.$
Moreover,   if the equation 
\begin{eqnarray*}\mathcal{A}(t,x)=\Phi(T^t(x))-\Phi(x). 
\end{eqnarray*}
holds,  then $\phi=\Phi'_{\xi}:= d\Phi(\xi).$

\section{Proof of Theorem \ref{Glivsic}.}

We now begin the proof of Theorem \ref{Glivsic}. First we state an essential definition. 
\begin{Def} Let $T$ be a  $\mathcal{C}^1$  Anosov volume-preserving diffeomorphism on 
a compact   Riemannian manifold
 $M$. For any given constants $C>0,\tilde{C}>0,\varepsilon>0$ and any given periodic point $p\in M$ with period $P(p)$, set 
 \begin{eqnarray*}
 \mathcal{F}_{T}(\tilde{C},\varepsilon,p)&=& \left\{ \phi\, |\,\phi \text{ is an $\alpha$-H\"{o}lder continuous function on } M, \int_M\phi\, dx=0,\right.\\
 &&\left.\,\,\,\, \,\,\,\,\,\,\Vert \phi\Vert _{\alpha}\leq \tilde{C}, \sum_{i=0}^{P(p)-1}\phi(T^i(p))\geq \varepsilon \right\}. 
 \end{eqnarray*}
 We say $T$  is of {\bf $(C,\tilde{C},\varepsilon,p)$-type}, if there exists a common time $N$, such that  for any $\phi\in \mathcal{F}_{T}(\tilde{C},\varepsilon,p)$, there exists at least one moment $1\leq k\leq N $ such that,  
 \begin{equation}\label{time}
\mu\{x\in M: \phi_k(x)> C\}>\frac{1}{2}-\varepsilon,
\end{equation}
where $\phi_k(x)=\sum_{i=0}^{k-1}\phi(T^i(x)).$
%$\exists A\subset M,s.t. vol(A)>1-\varepsilon$, $\forall x\in A$,
%$$\frac{\sharp \{1\leq k\leq n| \sum_{i=0}^{k-1}\phi(T^i(x))>C\}}{n}>\frac{1}{2}-\varepsilon$$ 
\end{Def}
In the following proposition, we use the Central Limit Theorem \ref{Central Limit Theorem} to prove that for any $(C,\tilde{C},\varepsilon,p)$,  $\mathcal{C}^2$ Anosov volume-preserving diffeomorphisms are $(C,\tilde{C},\varepsilon,p)$-type.
\begin{Prop}\label{property for C^2}Let $T$ be a  $\mathcal{C}^2$  Anosov volume-preserving diffeomorphism on a
compact    manifold
 $M$. For any  $(C,\tilde{C},\varepsilon,p)$, $T$ is  $(C,\tilde{C},\varepsilon,p)$-type.
\end{Prop}
\begin{proof}Fix constants $(C,\tilde{C},\varepsilon)$ and a periodic point $p$ arbitrarily.  
According to  Theorem \ref{Central Limit Theorem}, for any  $\phi\in\mathcal{F}_{T}(\tilde{C},\varepsilon,p)$, there exists $\sigma>0$, such that 
for any  $\alpha_0>0$, there exists $N_0\in
\mathbb{N}$ satisfying  for any $n\geq N_0$, 
\begin{eqnarray*}
m\left\{x\in M: \frac{\sum_{i=0}^{n-1}\phi(T^i(x))}{\sigma \sqrt{n}}>\alpha_0\right\}&\geq&
\frac{1}{\sqrt{2\pi}}\int_{\alpha_0}^{+\infty}e^{-\frac{1}{2}u^2}du-\frac{\varepsilon}{2} \\
&\geq& \frac{1}{2}e^{-\frac{1}{\sqrt{2}}\alpha_0}-\frac{\varepsilon}{2}
\end{eqnarray*}
Choose $\alpha_0$  small enough such that  $ \frac{1}{2}e^{-\frac{1}{\sqrt{2}}\alpha_0}\geq \frac{1}{2}-\frac{\varepsilon}{2}$. Assume
$N_1$ to be an integer satisfying $\frac{C}{\sigma \sqrt{N_1}}\leq\alpha_0$.  
Let $N(\phi):=\max\{N_0,N_1\}$. Then, for any $n\geq N(\phi)$, 
\begin{eqnarray*}m\left\{x\in M:\sum_{i=0}^{n-1}\phi(T^i(x))>C\right\}&\geq&m\left\{x\in M: \frac{\sum_{i=0}^{n-1}\phi(T^i(x))}{\sigma \sqrt{n}}>\alpha_0\right\}\\
&>& \frac{1}{2}e^{-\frac{1}{\sqrt{2}}\alpha_0}-\frac{\varepsilon}{2}\\
&>& \frac{1}{2}-\varepsilon .
\end{eqnarray*}

For this fixed time $N(\phi)$,  there exists a small neighborhood $\mathcal{U}(\phi)$ of $\phi$ such that   for any function $\tilde{\phi}\in \mathcal{U}(\phi)$,  we have   
  \begin{equation}\label{time}
m\left\{x\in M: \tilde{\phi}_{N(\phi)}(x)> C\right\}>\frac{1}{2}-\varepsilon,
\end{equation}
where $\tilde{\phi}_{N(\phi)}(x)=\sum_{i=0}^{N(\phi)-1}\tilde{\phi}(T^i(x)).$

  Due to the compactness of the set $\mathcal{F}(\tilde{C},\varepsilon,p)$,  there exists a finite
cover $\mathcal{P}=\{\mathcal{U}(\phi_i)\}_{i=0}^K$ of  $\mathcal{F}(\tilde{C},\varepsilon,p)$ and thus a common time $$N=\max_{1\leq i\leq K} N(\phi_i).$$
    This common time $N$ satisfies the condition we want. 
%Then, $n(\phi):=\frac{N(\alpha)}{\varepsilon}$ is the integer we want. 
\end{proof}
Now we prove that $(C,\tilde{C},\varepsilon,p)$-type implies Liv\v{s}ic measurable rigidity. 
\begin{Prop}\label{Type Good} Let $T$ be a  $\mathcal{C}^1$  Anosov volume-preserving diffeomorphism on 
a compact   Riemannian manifold
 $M$. Assume that for any $C>0$, $\tilde{C}>0$,  $\varepsilon>0$  and any periodic point $p$, $T$ is $(C,\tilde{C},\varepsilon,p)$-type.  Then  for any $\alpha$-H\"{o}lder continuous function $\phi:M\rightarrow \mathbb{R}$,  we have three equivalent  properties as follows:
 \begin{enumerate}
 \item[$(1)$] $\phi(x)=\Phi(T(x))-\Phi(x)$ has a continuous solution;
\item[$(2)$] $\sum_{x\in \mathcal{O}}\phi(x)=0$, for every $T$-periodic orbit $\mathcal{O}$;
\item[$(3)$] $\phi(x)=\Phi(T(x))-\Phi(x),a.e.$ for some measurable function $\Phi$.
\end{enumerate}

\end{Prop}
\begin{proof}By Theorem \ref{livsic for Anosov} and  the fact that $\mathcal{C}^1$ volume-preserving Anosov diffeomorphisms are transitive, we only need to check measurable rigidity, i.e. proving (1) from (3). Assume  $\Phi$ is a measurable solution to $\phi(x)=\Phi(T(x))-\Phi(x),a.e.$,  then  $\Phi$ is finite almost everywhere. For any small number $\varepsilon>0$, there exists 
a constant $C_{\varepsilon}>0$ such that
$$m\{x\in M:\Phi(x)\leq C_{\varepsilon}\}> 1-\frac{\varepsilon}{2}. $$ Thus, by the identity $\phi_n(x)=\Phi(T^n(x))-\Phi(x)$, it follows that
$$m\{x\in M: \phi_n(x)\leq 2C_{\varepsilon}\}>1-\frac{\varepsilon}{2}, \,\,\forall\,\, n\geq1.$$
If there is no continuous solution for $\phi(x)=\Phi(T(x))-\Phi(x)$, then there must exist a periodic point $p$ and $\varepsilon>0$ such that 
$\sum_{i=0}^{P(p)-1}\phi(T^i(p))\geq \varepsilon$.
However, as $T$ is $(C,\tilde{C},\varepsilon,p)$-type, for $C>2C_{\varepsilon}$ and  function $\phi$, there exists a time $1\leq k\leq N,$ such that
\begin{eqnarray*}(1-\frac{\varepsilon}{2})+(\frac{1}{2}-\varepsilon)&\leq& m\{x\in M: \phi_k(x)\leq 2C_{\varepsilon}\}+m\{x\in M: \phi_k(x)\geq C\}\\
&\leq& 1,
\end{eqnarray*}
which is a contradiction. So there exists continuous solution $\tilde{\Phi}:M\rightarrow \mathbb{R}$ such that  $\phi_n(x)=\tilde{\Phi}(T^n(x))-\tilde{\Phi}(x)$ and moreover 
$\Phi=\tilde{\Phi},a.e.$.
\end{proof}
From Proposition \ref{property for C^2}, we have $\mathcal{C}^2$  Anosov volume-preserving diffeomorphisms are good type. According to A. Avila's result \cite{AA} about the regularization of volume-preserving maps:
\begin{Thm}\label{Dense}\cite{AA} Smooth maps are  $\mathcal{C}^1$ dense in $\mathcal{C}^1$ volume-preserving maps.
\end{Thm}
By Theorem \ref{Dense} and the $\mathcal{C}^1$ stability of Anosov systems, we can get $\mathcal{C}^2$  Anosov volume-preserving diffeomorphisms are dense in $\mathcal{C}^1$  Anosov volume-preserving diffeomorphisms, which is a  key point in our proof. 
\begin{Thm}There exists a residual subset $\mathcal{G}$ of $\mathcal{C}^1$  Anosov volume-preserving diffeomorphisms on a compact    Riemannian manifold $M$ such that for any $T\in\mathcal{G}$ and any $\phi:M\rightarrow \mathbb{R}$ H$\ddot{o}$lder 
continuous,
the Liv\v{s}ic theorem holds, i.e. 
the following three conditions are equivalent:
\begin{enumerate}
 \item[$(1)$] $\phi(x)=\Phi(T(x))-\Phi(x)$ has a continuous solution $\Phi$;
\item[$(2)$] $\sum_{x\in \mathcal{O}}\phi(x)=0$, for every $T$-periodic orbit $\mathcal{O}$;
\item[$(3)$] $\phi(x)=\Phi(T(x))-\Phi(x),a.e.$ for some measurable function $\Phi$.
\end{enumerate}
\end{Thm}
\begin{proof}We only need to prove (1) from (3) generically. Take a countable basis $\mathcal{V}=\{\mathcal{V}_1,\mathcal{V}_2,\cdots\}$ of $M$. Let $\mathcal{A}_m^1$ be the set of $\mathcal{C}^1$ Anosov volume-preserving diffeomorphisms and let $\mathcal{A}_m^2$ be the set of $\mathcal{C}^2$ Anosov volume-preserving diffeomorphisms.  Denote \begin{eqnarray*}H_k&=&\{G\in \mathcal{A}_m^1\,|\, \text{there exists a neighborhood } U(G)\subset \mathcal{A}_m^1 \text{ such that for any }\\ && G_1\in U(G), 
\text{ for any } C>0,\tilde{C}>0,\varepsilon>0 \text{ and for any periodic point of } G_1, \\&& p \in 
\mathcal{V}_k\in\mathcal{V} \text{ with period } P(p)\leq k,
G_1 \text{ is } (C,\tilde{C},\varepsilon, p)-\text{type} \}.
\end{eqnarray*} 
 It is easy to see that $H_k$ is open. Set $$\mathcal{G}:=\cap_{k\in \mathbb{N}}H_k.$$ 
Now we prove  $\mathcal{G}$ is the generic set we want.   

Let $\mathcal{A}_m^2$ be the set of $\mathcal{C}^2$ Anosov volume-preserving diffeomorphisms. 
In order to proof the density of $H_k$, we prove the following lemma first.

\begin{Lem}\label{middle lemma} The set $\mathcal{A}_m^2$ is contained in $H_k$, for all $k\geq 1.$
\end{Lem}
\begin{proof}

  Fix any $T\in\mathcal{A}_m^2$ and $\mathcal{V}_k.$ We finish our proof by choosing smaller and smaller neighborhoods of $T$.  
By Proposition \ref{property for C^2}, for any  $T\in\mathcal{A}_m^2$ and  any $(C,\tilde{C},\varepsilon,p)$, $T$ is  $(C,\tilde{C},\varepsilon,p)$-type. Thus, there exists $N$ such that :
  for any $\phi\in\mathcal{F}_{T}(\tilde{C},\varepsilon,p)$,  there exists $1\leq i\leq N$ such that, 
  \begin{equation*}
  m\{x\in M: \phi_i(x)> C\}>\frac{1}{2}-\varepsilon.
  \end{equation*}

Consider the set 
\begin{eqnarray*}
 \mathcal{F}_{G}(\tilde{C},\varepsilon,p)&=&\left\{ \phi\, |\,\phi \text{ is an $\alpha$-H\"{o}lder continuous function on } M, \int_M\phi \,dx=0,\right.\\
 &&\,\,\,\, \,\,\,\,\,\,\left.\Vert \phi\Vert _{\alpha}\leq \tilde{C}, \sum_{i=0}^{P(p)-1}\phi(G^i(p))\geq \varepsilon\right\}, 
 \end{eqnarray*}
 where $G$ is a $\mathcal{C}^1$ Anosov volume-preserving diffeomorphism  $\mathcal{C}^1$-close to $T$ and $p\in \mathcal{V}_k$ is a periodic point with period $P(p)\leq k$ for $G$.

There exists a small neighborhood $W(T)$ of  $T$ such that for any $G\in W(T)$,  $\mathcal{F}_{G}(\tilde{C},\varepsilon,p)\subset \mathcal{F}_{T}(\tilde{C},\frac{\varepsilon}{2},q)$, where $q$ is the continuation of  $p$ given by structure stability with the same period $P(p)\leq k$. Thus, there exists $N$ such that for any $\phi\in\mathcal{F}_{G}(\tilde{C},\varepsilon,p)\subset \mathcal{F}_{T}(\tilde{C},\frac{\varepsilon}{2},q),$   
we have a time $1\leq i\leq N$ such that, 
  \begin{equation*}
  m\{x\in M: \phi_{i,T}(x)> C\}>\frac{1}{2}-\frac{\varepsilon}{2},
  \end{equation*}
  where $\phi_{i,T}=\sum_{j=0}^{i-1}\phi(T^j(x)).$
 Next, there exists a smaller neighborhood $V(T)\subset W(T)$ of $T$ such that for any $G\in V(T)$, there exists $N$ such that for any $\phi\in\mathcal{F}_{G}(\tilde{C},\varepsilon,p)\subset \mathcal{F}_{T}(\tilde{C},\frac{\varepsilon}{2},q),$  we have 
there exists $1\leq i\leq N$ such that, 
  \begin{equation*}
  m\left\{x\in M: \phi_{i,G}(x)> \frac{C}{2}\right\}>\frac{1}{2}-\frac{\varepsilon}{2}.
  \end{equation*}

Taking the uniform hyperbolicity of $T$ into account, there are only finite $p\in \mathcal{V}_k$ with period $P(p)\leq k$ for every $G\in V(T).$  Thus we get another smaller neighborhood $U(T)\subset V(T),$ such that for any $G\in U(T)$, and any constants $C>0, \tilde{C}>0, \varepsilon>0$ and any periodic point $p\in \mathcal{V}_k$ with period $P(p)\leq k$,  $G$ is $(C,\tilde{C},\varepsilon,p)$-type. 

Thus, $T\in H_k$. This completes the proof of this lemma.
\end{proof} 
 
So Lemma \ref{middle lemma}  implies  that $H_k$ is $\mathcal{C}^1$ dense in $\mathcal{C}^1$ Anosov volume-preserving diffeomorphisms and then we  get that $\mathcal{G}$ is a generic set. 

It is easy to see from the definition that  for every diffeomorphism $G\in\mathcal{G}$ and  tuple $(C,\tilde{C},\varepsilon,p)$,  $G$ is $(C,\tilde{C},\varepsilon,p)$-type.  By Proposition \ref{Type Good}, we finish the proof. 

\end{proof}

\section{Proof of Theorem \ref{Flow}}
The argument for Anosov flows proceeds in an almost identical fashion as in the previous section, {\em mutatis mutandis}.   Theorem \ref{livsic for flow} and Theorem \ref{C2}, instead of Theorem \ref{livsic for Anosov} and Theorem \ref{Central Limit Theorem}, are needed in the proof of  Theorem \ref{Flow}.

\end{document}